\documentclass[12pt]{article}
\usepackage{amssymb,amsfonts,amsmath}
\usepackage{graphics}
\usepackage{graphicx}
\usepackage{epsfig}
\usepackage{color}
\usepackage{enumerate}
\usepackage{amsthm}

\newcommand\CC{{\mathbb C}}

\newcommand\DD{{\mathbb D}}
\newcommand\NN{{\mathbb N}}
\newcommand\RR{{\mathbb R}}

\newcommand\TT{{\mathbb T}}
\newcommand\LH{{\mathcal L}({\cal H})}

\newcommand\ZZ{{\mathbb Z}}

\def\Im{\mathop{\rm  Im}\nolimits}

\def\beq{\begin{equation}}
\def\eeq{\end{equation}}
\newtheorem{thm}{Theorem}[section]
\newtheorem{prop}[thm]{Proposition}
\newtheorem{defn}[thm]{Definition}
\newtheorem{lem}[thm]{Lemma}
\newtheorem{cor}[thm]{Corollary}
\newtheorem{rem}[thm]{Remark}

\newtheorem{ex}[thm]{Example}

\newcommand\beginpf{\noindent {\bf Proof:} \quad}

\newcommand\re{\mathop{\rm Re}\nolimits}

\newcommand\esssup{\mathop{\rm ess\ sup}\nolimits}

\def\beginpf{\begin{proof}}
\def\endpf{\end{proof}}

\def\N{{\mathbb N}}
\def\R{{\mathbb R}}
\def\T{{\mathbb T}}

\newcommand\IM{\mathop{\rm Im}\nolimits}
\renewcommand\phi{\varphi}

\newcommand\D{\mathcal {D}}

\title{Analyticity and compactness of semigroups of composition operators}
\author{ C. Avicou\thanks{I. C. J., UFR de
Math\'ematiques,  Universit\'e  Lyon~1, 43 bld. du 11/11/1918,
69622 Villeurbanne Cedex, France.
  \protect\linebreak[3]
{\tt avicou@math.univ-lyon1.fr}.},
 I. Chalendar\thanks{I. C. J., UFR de
Math\'ematiques,  Universit\'e  Lyon~1, 43 bld. du 11/11/1918,
69622 Villeurbanne Cedex, France.
  \protect\linebreak[3]
{\tt chalendar@}math.univ-lyon1.fr}
\ and   J.R. Partington\thanks{School of
Mathematics,
University of Leeds, Leeds LS2 9JT, U.K.
\protect\linebreak[3]
{\tt J.R.Partington@leeds.ac.uk}.}
}

\begin{document}

\maketitle
\begin{abstract}
	This paper provides  a complete characterization of quasicontractive   groups and analytic $C_0$-semigroups on  
	Hardy and  Dirichlet space on the unit disc with a prescribed generator of the form $Af=Gf'$. 
	In the analytic case we also give a complete characterization 
	of immediately compact semigroups. When the analyticity fails, we obtain sufficient conditions for 
	compactness and membership in the trace class. Finally, we analyse the case where the unit disc is replaced by 
	the right-half plane, where the results are drastically different.     
\end{abstract}



\section{Introduction}

Semigroups of composition operators acting on the Hardy space $H^2(\DD)$ or the
Dirichlet space $\D$ have been extensively studied 
(see, for example, \cite{Arv, ACP, BP,koenig,sis1, siskakis}).

These are associated with the notion of semiflow $(\phi_t)$ of analytic functions mapping
the unit disc
$\DD$ to itself, and satisfying $\phi_{s+t}=\phi_s \circ \phi_t$; here $s$ and $t$
lie either in $\RR_+$ or in a sector of the complex plane. 
It is assumed that the mapping $(t,z) \mapsto \phi_t(z)$ is jointly continuous.
It follows that there exists an analytic  function $G$ on $\DD$ such that 
\[
\frac{\partial \phi_t}{\partial t}=G \circ \phi_t.
\]

A semiflow
induces
composition operators $C_{\phi_t}$ on $H^2(\DD)$ or $\D$, where
$C_{\phi_t}f = f \circ \phi_t$. If it is strongly continuous, then it has a
densely-defined generator $A$ given by $Af=Gf'$, with $G$ as above.
Fuller details are given later.

In Section \ref{sec:2} we give a characterization of analytic semigroups in terms of
the properties of $G$, using the complex Lumer--Phillips theorem \cite{ABHN} (this is
appropriate, since the semigroup is quasicontractive, as explained below). In addition,
 we give a complete description of groups of composition operators in terms of the function $G$.

The theme of Section \ref{sec:3} is compactness, together with Hilbert--Schmidt and trace-class
properties. For example, we give sufficient conditions on $G$ for the semigroup to
be immediately compact; these are necessary and sufficient (and equivalent to
eventual compactness) when the semigroup is analytic.
We give  examples to illustrate the various possibilities involving the properties of
immediate compactness and eventual compactness. 
Although most of our results are obtained in terms of the properties of $G$, we are also
able to derive results on compactness from the semiflow model $\phi_t(z)=h^{-1}(e^{-ct}h(z))$.
In particular we are able to 
provide some answers to a question raised by Siskakis
 \cite[Sec. 8]{siskakis} about how the behaviour of such semigroups
depends on the properties of $h$.

Section \ref{sec:4} is concerned with analytic semigroups and groups 
of composition operators on the half-plane. Such operators are
never compact.

\section{Analytic semigroups and groups of composition operators}
\label{sec:2}

\begin{defn}
Let $(\beta_n)_{n \ge 0}$ be a sequence of positive real numbers.
Then $H^2(\beta)$ is the space of analytic functions 
\[
f(z)=\sum_{n=0}^\infty c_n z^n
\]
 in the unit disc $\DD$ that have finite norm
\[
\|f\|_\beta= \left( \sum_{n=0}^\infty |c_n|^2 \beta_n^2 \right)^{1/2}.
\]
The case $\beta_n=1$ gives the usual Hardy space $H^2 (\DD)$. \\
The case $\beta_0=1$ and $\beta_n=\sqrt{n}$ for $n\geq 1$ provides the Dirichlet space
$\cal D$, which is included in $H^2(\DD)$. \\
The case $\beta_n=1/\sqrt{n+1}$ produces the Bergman space,
which contains $H^2(\DD)$.
\end{defn}

\subsection{General properties of semigroups}

A $C_0$-semigroup $(T(t))_{t \ge 0}$
on a Banach space $X$ is a mapping $T:\RR_+ \to \mathcal L(X)$ satisfying
\[\begin{cases}
 & T(0)=I,\\
\forall t,s \geqslant 0, &T(t+s)=T(t)\circ T(s),\\
\forall x\in X, & \lim_{t\to 0} T(t)x=x.
\end{cases}\]
A consequence of this definition is the existence of two scalars $w\geqslant 0$ and $M\geqslant 1$ such that for all $t\in \R_+$, $\|T(t)\|\leqslant Me^{wt}$. In particular, if $M=1$, the semigroup $T$ is said to be quasicontractive. If in addition $w=0$, $T$ is a 
contractive semigroup.

A $C_0$-semigroup $T$ will be called analytic (or holomorphic) if there exists a sector $\Sigma_\theta=\{ re^{i\alpha}, r\in\R_+, |\alpha|<\theta\}$ with $\theta \in (0,\frac\pi 2]$ and an analytic mapping $\widetilde T:\Sigma_\theta \to \mathcal L(X)$ such that $\widetilde T$ is an extension of $T$ and
\[\sup_{\xi\in \Sigma_\theta \cap \DD} \| \widetilde T (\xi)\| <\infty.\]
In both cases, the generator of $T$ (or $\widetilde T$) will be the linear operator $A$ defined by
\[
D(A)=\left\{x\in X, \lim_{\R\ni t\to 0} \frac{T(t)x-x}{t} \text{ exists}\right\}\]
and, for all $x \in D(A)$,
\[  Ax=\lim_{\R\ni t\to 0} \frac{T(t)x-x}{t}.\]

Recall that an operator $A$ is dissipative if $\re \langle Ax,x\rangle \le 0$ for $x \in D(A)$.
The classical Lumer--Phillips theorem
asserts that $A$
 generates a contraction semigroup if and only if
$A$ is dissipative and $ I - A$ is surjective (see, for example, \cite[Thm. 3.4.5]{ABHN}).

The following extension of this to analytic semigroups is given in \cite{arendt}.

\begin{prop}\label{prop:arendtelst}
Let
$A$ be an operator on a complex Hilbert space $H$ and let $\theta \in (0, \pi/2 )$. The
following are equivalent.

(i) $A$ generates an analytic $C_0$-semigroup which is contractive
on the sector $\Sigma_\theta$;

(ii) $e^{\pm i\theta} A$ is dissipative and $I - A$ is surjective.
\end{prop}

From this we have the following corollary, which appears to be new.

\begin{cor}
Suppose that $A$ is an operator on a Hilbert space and $\theta \in (0,\pi/2)$. If 
$A$ and $\pm e^{i\theta}A$ generate quasicontractive semigroups, then $A$ generates
an analytic semigroup on the sector $\Sigma_\theta$.
\end{cor}

\beginpf
There exist $\delta_1,\delta_2,\delta_3 \ge 0$ such that $A-\delta_1 I$, $e^{i\theta}A-\delta_2 I$ and
$e^{-i\theta} A-\delta_3 I$ are all dissipative.

It follows that $A-\alpha I$, $e^{i\theta}(A-\alpha I)$ and $e^{-i\theta}(A-\alpha I)$ are all dissipative
provided that $\alpha \ge \max\{ \delta_1, \delta_2/\cos\theta, \delta_2/\cos\theta \}$. Moreover,
$I - (A-\alpha I)$ is surjective, and so the result follows from Proposition~\ref{prop:arendtelst}.

\endpf

\subsection{An algebraic characterization of composition operators}

The following characterization will be useful
in order to show that an analytic semigroup
consists of composition operators whenever its restriction to $\RR_+$ has this property.

In \cite[Thm. 5.1.13]{rosenthal} it is shown that a bounded linear operator $T$ on $H^2(\DD)$ is a composition
operator if and only if, for the functions $e_n: z \mapsto z^n$, we have $Te_n=(Te_1)^n$ for all $n \in \NN$.
A similar characterization holds in the weighted Hardy space $H^2(\beta)$, with one
supplementary condition.

\begin{prop}
Let $T : H^2(\beta) \to H^2(\beta)$ be a bounded linear operator. The operator $T$ is a composition operator if and only if
both $Te_1(\DD)\subset \DD$ and for all $n\in\N$, $Te_n=(Te_1)^n$.
\end{prop}

\beginpf
If $T=C_\varphi$ is a composition operator, then $Te_1=\varphi :\DD \to\DD$ and $Te_n=\varphi^n=(Te_1)^n$ for all $n\in\N$.

Conversely, we note that $\varphi = Te_1\in H^2(\beta)$. The function $\varphi$ is analytic and maps $\DD$ to $\DD$. Besides, for every $n\in\N$, $Te_n=\varphi^n=C_\varphi e_n$. Thus, the linearity and the continuity of $T$ and $C_\varphi$ imply that $T=C_\varphi$.
\end{proof}

We require this for the following result, which applies in particular to the Hardy and Dirichlet spaces.

\begin{cor}\label{cor:dir}
Let $T=(T(t))_{t \ge 0}$ be a $C_0$-semigroup of composition operators on $H^2(\beta)$,
where $\beta_n=O(\sqrt n)$. If $T$ has an analytic extension to a sector  $\Sigma_\theta$, then for every $\xi\in\Sigma_\theta$, $T(\xi)$ is composition operator.
\end{cor}

\begin{proof}
Let $n\in\N$. We define $f_n : \Sigma_\theta \to  H^2(\DD),\  \xi \mapsto T(\xi)e_n-(T(\xi)e_1)^n.$ As $T(t)$ is a composition operator for each $t\in\R_+$, the function $f_n$ is zero on $\R_+$. Thus by analyticity of $f_n$ on $\Sigma_\theta$, $f_n \equiv 0$.

It remains to check that for all $\xi\in\Sigma_\theta$, $T(\xi)e_1(\DD) \subset \DD$. Supposing that this is not true, then
there exists $\alpha \in\DD$ such that $|T(\xi)e_1(\alpha)|\ge 1$. Since $T(\xi)e_1$ is analytic, then  
we can suppose that
either we have $|T(\xi)e_1(\alpha)|>1$ or that $T(\xi)e_1$ is a constant $\lambda$	of modulus 1. In the first case
 $|T(\xi)e_n(\alpha)|/\|e_n\|=|T(\xi)e_1(\alpha)|^n/\|e_n\|  \to \infty$, as $n \to \infty$ and this contradicts the boundedness  
of $T(\xi)$. In the second case, $T(\xi)$ maps $e_n$ to $\lambda^n$ (including $n=0$), and thus, if it were bounded on $H^2(\beta)$,
it would be given as the inner product with the function $\sum_{k=0}^\infty \lambda^k z^k/ \beta_k^2$. However,
this function does not lie in $H^2(\beta)$ as the square of its norm would be $\sum_{k=0}^\infty 1/\beta_k^2$, which diverges.
\end{proof}

A similar characterization holds for weighted composition operators.

\begin{thm}\label{thm:wcompo}
	Let $T:H^2(\DD)\to H^2(\DD)$ be a linear and bounded operator. 
	Assume that  $Te_0\not\equiv 0$ and $Te_0\in H^\infty(\DD)$. Then,
	$T$ is a weighted composition operator if and only if  $(Te_0)^{n-1}Te_n=(Te_1)^n$ for all positive integers $n$.
\end{thm}
\beginpf
If $T$ is the weighted composition operator $M_wC_\phi$ defined by 
	$Tf=w f\circ\varphi$, then it follows that $Te_0=w$ and $Te_1=w\varphi$.  Therefore, for all $n\geq 1$, $(Te_0)^{n-1}Te_n=(Te_1)^n$ is satisfied.\\
	
Conversely, assume that for every $n\geq 1$, $(Te_0)^{n-1}Te_n=(Te_1)^n$. Let $w=Te_0\in H^2(\DD)$. Since $Te_0$ is not identically zero, the set $Z$ of its zeroes is discrete. For $z\in\DD\backslash Z$, let $\varphi (z)=\frac{Te_1(z)}{Te_0(z)}$. It remains to check that $\varphi(\DD\backslash Z)\subset\DD$ for  $z\in\DD\backslash Z$.
	
Assume that there exists $z_0\in\DD\backslash Z$ such that $|\varphi(z_0)|>1$. Then,
	\[|w(z_0)||\varphi^n(z_0)|=|\langle Te_n,k_{z_0}\rangle|\leqslant \| T\| \|k_{z_0}\|,\] which contradicts $|\varphi^n(z_0)|\rightarrow \infty$.
	
Assume now the existence of $z_0\in\DD\backslash Z$ such that $|\varphi(z_0)|=1$. By the maximum principle $\varphi(z)=\lambda\in\T$ for every $z\in\DD\backslash Z$. Thus, $Te_n=\lambda^n w =\lambda^n Te_0$ for all $n\in\NN$. Hence we get  \[\|T^*Te_0\|^2=\sum_{n\in\NN}|\langle Te_0,Te_n\rangle|^2=\sum_{n\in\NN}|\lambda|^{2n}=\infty,\]
	a contradiction.
	Thus we obtain $|\varphi(z)|<1$ and $Te_n=w\varphi^n$ for all $n\in\NN$. If $Te_0 \in H^\infty (\DD)$, then the continuity of $M_w C_\varphi$ implies that $T=M_w C_\varphi$.
\endpf
 
\begin{cor}
	Let $T(t)$ be a $C_0$-semigroup of weighted composition operators on $H^2(\DD)$. If $T(t)$ has an analytic extension to a sector $\Sigma_\theta$, then for every $\xi\in\Sigma_\theta$, $T(\xi)$ is a weighted composition operator.
\end{cor}

\beginpf
Fix $n\geq 1$. Define $f_n :  \Sigma_\theta \to  H^2(\DD)$
by 
\[
f_n(\xi)= (T(\xi)e_0)^{n-1}T(\xi)e_n-(T(\xi)e_1)^n.
\]
 Since $T(t)$ is a weighted composition operator for each $t>0$, the function $f_n$ vanishes on $\RR^+$. Thus, the analyticity of $f_n$ on $\Sigma_\theta$ implies that $f_n \equiv 0$.\\
	If $Te_0\equiv 0$, then $T(t)$ is trivial. Otherwise, the semigroup $T$ being analytic, $\sup_{ \DD\cap \Sigma_\theta} \|T(\xi)\| < +\infty$. It follows that 
	for all $\xi \in \DD\cap \Sigma_\theta$, $\| T(\xi)e_0\| \leqslant M$; i.e.,
	for all $\xi \in n\DD\cap \Sigma_\theta$, $\| T(\xi)e_0\| \leqslant M^n$. Thus,
	for all $\xi \in  \Sigma_\theta$, $ T(\xi)e_0\in H^\infty(\DD)$. The conclusion follows from Theorem~\ref{thm:wcompo}.
\endpf

\subsection{Quasicontractive analytic semigroups on the Hardy and Dirichlet space}

In order to characterise quasicontractive analytic semigroups in terms of the 
associated function $G$ we begin with the following result. Note that here and elsewhere we
use \cite[Thm.~3.9]{ACP}, which makes the hypothesis that $G \in H^2(\DD)$. This hypothesis ensures that the
generator has dense domain, but is not necessary, as, for example the case
$G(z)=-z/(z+1)$ shows: here $D(A)$ contains $(z+1)^2\CC[z]$, which is dense in $H^2(\DD)$.

\begin{thm}
 Let $G: \DD \to \CC$ be a holomorphic function such that the operator  $A$ 
defined by $Af(z)=G(z)f'(z)$ has dense domain $D(A)\subset H^2(\DD)$  (resp. $D(A)\subset \D$). Then the following are equivalent:
\begin{enumerate}
\item The operator $A$ generates a quasicontractive analytic semigroup on  $H^2(\DD)$ (resp. $\D$).
\item The operator $A$ generates an analytic semigroup of composition operators on $H^2(\DD)$ (resp. $\D$).
\item There exists $\theta \in (0,\frac \pi 2)$ such that the operators $e^{i\theta}A$, $e^{-i\theta}A$ 
and $A$ generate $C_0$-semigroups of composition operators on $H^2(\DD)$ (resp. $\D$).
\item There exists $\theta \in (0,\frac \pi 2)$ such that
\[ \sup \left\{ \re  \langle e^{\pm i\theta} Af,f\rangle : f\in D(A), \|f\|=1\right\}<\infty,\]
and $\lambda>0$ such that
\[(A-\lambda I)D(A)=H^2(\DD) \  (\hbox{resp. }\D).\]
\end{enumerate}
\end{thm}

\begin{proof}~
\begin{itemize}
\item[1.$\Rightarrow$ 2.] We denote by $(T(t))$ the semigroup on the sector $\Sigma_\theta$ generated by the operator $A$. From \cite[Theorem 3.9]{ACP}, the restriction to $\R_+$ of this semigroup is a semigroup of composition operators. 
Now by Corollary~\ref{cor:dir} it follows that $(T(t))$ consists of composition operators. 

\item[2.$\Rightarrow$ 3.] This is immediate.

\item[3.$\Rightarrow$ 4.] 
By \cite[Theorem 3.9]{ACP} any $C_0$-semigroup of composition operators is quasi\-contractive. The
result now follows from
 the Lumer--Phillips theorem. 

\item[4.$\Rightarrow$ 1.] This follows from an obvious corollary of Proposition~\ref{prop:arendtelst}.
\end{itemize}
\end{proof}

If $G$ generates a semiflow of analytic functions on $\DD$, then
$G$ has an expression of the form $G(z)= (\alpha-z)(1-\overline\alpha z)F(z)$, where $\alpha\in \overline\DD$ and
$F: \DD \to \CC_+$ is holomorphic (see \cite{BP}). In particular, $G$ has radial limits 
almost everywhere on $\TT$, since $F$ is the composition of a M\"obius mapping and a function in $H^\infty(\DD)$.
Note that this applies to every semigroup of composition operators, independently of the
underlying Hilbert function space, since it is associated with a semiflow.
In \cite{ACP} it is shown that $A: f \mapsto Gf'$ generates 
 a $C_0$-semigroup of composition operators on $H^2(\DD)$  or $\D$ if and only if
$\operatorname{ess}\sup_{z\in\T} \re \overline z G(z) \le 0$.
As before, it is not necessary to assume that $G \in H^2(\DD)$.
This can now be applied to give easy conditions on the same operator $A$ for the case of analytic semigroups.

\begin{cor}\label{cor:analytic}
 Let $G: \DD \to \CC$ be a holomorphic function such that the operator  $A$ 
defined by $Af(z)=G(z)f'(z)$ has dense domain $D(A)\subset H^2(\DD)$  (resp. $D(A)\subset \D$). 
The operator $A$ generates an analytic semigroup of composition operators on $H^2(\DD)$ (resp. $\D$) if and only if there exists $ \theta \in (0,\frac \pi 2)$ such that for all $\alpha \in \{-\theta,0,\theta\}$
\[\operatorname{ess}\sup_{z\in\T} \re(e^{i\alpha}\overline z G(z)) \le  0.\]
\end{cor}

Geometrically, this condition says that the image of $\TT$ under $z \mapsto \overline z G(z)$ is
contained in a sector $-\Sigma_{(\frac{\pi}{2}-\theta)}$ in the left half-plane.

\subsection{Groups of composition operators}

The following remark enables one to characterize  groups of composition operators on $H^2(\DD)$
and $\D$.
\begin{prop}\label{prop:group}
 Let $G: \DD \to \CC$ be a holomorphic function such that the operator  $A$ 
defined by $Af(z)=G(z)f'(z)$ has dense domain $D(A)\subset H^2(\DD)$  (resp. $D(A)\subset \D$).  The following are equivalent.
\begin{enumerate}
\item The operator $A$ generates a $C_0$-group of composition operators.
\item $  \re (\overline{z}G(z))=0$ almost everywhere on $\TT$.
\end{enumerate}
\end{prop}

\beginpf
The result follows directly  from \cite[Thm.~3.9]{ACP}
 since $A$ generates a $C_0$-group of composition operators if and only if 
both $A$ and $-A$ generate $C_0$-semigroups of composition operators.
\endpf

\begin{cor}
The only analytic group of composition operators on $H^2(\DD)$   or $\D$  is the trivial semigroup.
\end{cor}

\beginpf
This follows immediately from Corollary~\ref{cor:analytic} and Proposition~\ref{prop:group}.
\endpf

This corollary can be seen as a consequence of more general results: an analytic
group is norm-continuous at $0$, and so its generator is bounded (see \cite{SF}).  However,
a non-trivial group of composition operators never has a bounded generator.

\begin{ex}
The semigroup $(C_{\varphi_t})$ of composition operators is a group if and only if for one (and thus any) $t_0\in\R_+$, the operator $C_{\varphi_{t_0}}$ is invertible: thus, if and only if $(\varphi_t)$ is a group of automorphisms. In that case, the   group will satisfy $\varphi_t^{-1}=\varphi_{-t}$. Considering the automorphism semigroup given by \[\varphi_t(z)=\frac{z+\tanh t}{1+z\tanh t}\] with generator $G(z)=1-z^2$, we   see  that $\overline z G(z)=-2i \IM(z) \in i\R$. Thus the given condition is satisfied.
\end{ex}

The easy example $G(z)=z(z-1)$ shows that it is possible for a $C_0$-semigroup of composition
operators to be neither analytic nor a group.

\section{Compactness of semigroups}
\label{sec:3}
\subsection{Immediate and eventual compactness}

We recall that a semigroup $(T(t))_{t \ge 0}$ is said to be {\em immediately compact}    
if
the operators $T(t)$ are compact   for all $t>0$. 
A semigroup $(T(t))_{t \ge 0}$ is said to be {\em eventually compact} if there exists $t_0>0$ such that $T(t)$ is compact for all $t\geq t_0$.  
Similar definitions hold for immediately/eventually Hilbert--Schmidt and trace-class.

We begin with an elementary observation.

\begin{prop}\label{prop:notimcomp}
Suppose that for some $t_0>0$ one has $|\phi_{t_0}(\xi)|=1$ on a set of positive measure; then
$C_{\phi_{t_0}}$ is not compact on $H^2(\DD)$ or $\D$, and so
the semigroup $(C_{\phi_t})_{t \ge 0}$ is  not immediately compact.
\end{prop}

\beginpf
For the Hardy space, this follows since the weakly null sequence
$(e_n)_{n \ge 0}$ with $e_n(z)=z^n$ is mapped into $(\phi_{t_0}^n)$, which does not converge to
$0$ in norm.
For the Dirichlet space the result follows from \cite[Ex.~6.3]{primer}.
\endpf

A slightly stronger result can be shown for the Hardy space,
using the following theorem \cite[Chap. 2, Thm 3.3]{pazy}, which links
immediate compactness with continuity in norm.
\begin{thm}\label{thm:cpresolv}
Let $(T(t))_{t \ge 0}$ be a $C_0$-semigroup and let $A$ be its infinitesimal generator.
Then $(T(t))_{t \ge 0}$ is immediately compact if and only if\\ (i)~the resolvent $R(\lambda,A)$ is
compact for all (or for one) $\lambda \in \CC \setminus \sigma(A)$, and\\ 
(ii)~$\lim_{s \to t} \|T(s)-T(t)\| =0$ for all $t>0$.
\end{thm}

Combining this with the following result due to Berkson \cite{berkson}, we 
see that, under the hypotheses of
of Proposition \ref{prop:notimcomp}, in the case of the Hardy space, the semigroup is not
norm-continuous and hence not immediately compact.

\begin{thm}
Let $\phi: \DD \to \DD$ be analytic. If $m\{\xi \in \TT: |\phi(\xi)|=1\}=\delta>0$, then
considered as operators on $H^2(\DD)$, we have
$\|C_\phi-C_\psi\| \ge \sqrt{\delta/2}$ for all $\psi: \DD \to \DD$ analytic with $\psi \ne \phi$.
\end{thm}

We shall now give a sufficient condition for immediate compactness of a semigroup of composition
operators, in terms of the associated function $G$. First, we recall a classical necessary and
sufficient condition for compactness of a composition operator $C_\phi$  in the case when
$\phi$ is univalent \cite[pp. 132, 139]{CM}.

\begin{thm}
For $\phi: \DD \to \DD$ analytic and univalent, the composition operator $C_\phi$ is compact
on $H^2(\DD)$
if and only if
\[
\lim_{z \to \xi} \frac{1-|z|^2} {1-|\phi(z)|^2} = 0
\]
for all $\xi \in \TT$.
\end{thm}

The following proposition collects together standard results on Hilbert--Schmidt and
trace-class composition operators.

\begin{prop}\label{prop:hs}
For $\phi:\DD\to\DD$ analytic  with $\|\phi\|_\infty<1$, the composition operator $C_\phi$ is 
trace-class on $H^2(\DD)$ \cite[p.~149]{CM};
if in addition $\phi \in \D$, then $C_\phi$ is Hilbert--Schmidt on $\D$ \cite[Cor.~6.3.3]{primer}.
\end{prop}

Siskakis \cite{siskakis} has given sufficient conditions for compactness of the resolvent
operator $R(\lambda,A)$ (which is a necessary condition for the immediate compactness of the
semigroup), in the case $G(z)=-zF(z)$, although they are not necessary, as the case
$G(z)=-z$ illustrates. 

\begin{thm}\label{thm:immcompact}
Let  $\delta>0$; suppose that there exists $\epsilon>0$ such that
\[
\re(\overline z G(z)) \le -\delta \qquad \hbox{for all} \quad z \quad \hbox{with} \quad 1-\epsilon < |z| < 1.
\]
Then $A$, defined by $Af=Gf'$, generates an immediately compact semigroup
of composition operators on $H^2(\DD)$ 
and $\D$. Indeed the semigroup is immediately trace-class. 
\end{thm}

\beginpf
For $t>0$ and $z \in \DD$ we have
\begin{eqnarray*}
\frac{\partial}{\partial t} |\phi_t(z)|^2 &=& 2 \re \left( \overline{\phi_t(z)} \frac{\partial}{\partial t} \phi_t(z) \right) \\
&=& 2 \re \left( \overline{\phi_t(z)} G(\phi_t(z)) \right)
\end{eqnarray*}
Hence 
\[
\frac{\partial}{\partial t} |\phi_t(z)|^2  \le  -2\delta
\]
whenever $1-\epsilon < |\phi_t(z)| < 1$. 
Also, by compactness, there exists an $M>0$ such that
\[
\frac{\partial}{\partial t} |\phi_t(z)|^2 \le M
\]
whenever $ |\phi_t(z)| \le 1-\epsilon$.

Choose $t_0 > 0$ such that
\[
a:=(1-\epsilon)^2+t_0 M < 1.
\]
For a fixed $z \in \DD$ we consider $\phi_t(z)$ over the interval $[0,t_0]$. Suppose first that $|\phi_t(z)| > (1-\epsilon)$ for all
$t \in [0,t_0]$. 
Then $|\phi_t(z)|^2 \le |z|^2 - 2\delta t $ for all $t \in [0,t_0]$.

Otherwise let 
\[
t_1=\inf \left \{t \in [0,t_0]: |\phi_t(z)| \le 1- \epsilon\right\}.
\]
Then $|\phi_t(z)|^2 \le a$ for all $t \in [t_1,t_0]$ and $|\phi_t(z)|^2 \le |z|^2 - 2\delta t$ for all $t \in [0,t_1]$.

Thus $\|\phi_t\|^2_\infty \le \max\{1-2\delta t, a\}<1$ for $0 < t \le t_0$ and hence $C_{\phi_t}$ is
Hilbert--Schmidt
for these $t$, 
by Propositionı~\ref{prop:hs}, and hence 
trace-class 
for all $t>0$ since $C_{\phi_t}=C_{\phi_{t/2}}^2$.
\endpf

\begin{cor}
Let $\eta>0$; suppose that
$\esssup_{z \in \TT} \re(\overline z G(z)) \le -\eta$, 
and
that $\re G'$ is bounded on $\DD$. 
Then $A$, defined by $Af=Gf'$, generates an immediately trace-class semigroup
of composition operators on $H^2(\DD)$ 
and $\D$. 
\end{cor}
\beginpf
Define $K(z):=G(z)+\eta z$, so that $\esssup_{z \in \TT}\re (\overline z K(z)) \le 0$. By
\cite[Thm. 4.3]{ACP}, it follows that
\[
2\re (\overline z K(z) )+ (1-|z|^2) \re K'(z) \le 0 \qquad (z \in \DD),
\]
and hence
\[
2\re (\overline z G(z) )+ (1-|z|^2) \re G'(z) \le -\eta(1+|z|^2) \le -\eta \qquad (z \in \DD).
\]
Now if $\| \re G'(z) \|_\infty \le M$, 
then we have $\re (\overline z G(z) ) \le -\eta$ whenever
$ |z| \ge 1- \dfrac{\eta}{2M}$.
The result now follows 
from Theorem \ref{thm:immcompact}.
\endpf

Easy examples of the above are  $G(z)=-z$ and $G(z)=z(z^2-2)$. However, Siskakis \cite{siskakis} gives the example
$G(z)=(1-z) \log (1-z)$, where the semigroup is immediately compact while   $\esssup \re \overline z G(z)=0$ on $\TT$.

\begin{rem}\label{rem:dw}{\rm
Note that all the examples of immediately compact semigroups have the Denjoy--Wolff point of $\phi_t$ in the open disc 
$\DD$.  This is always the case: for
let $\phi: \DD \to \DD$ be analytic, such that
$C_\phi$ is a compact composition operator on $H^2(\DD)$. Then the Denjoy--Wolff point of $\phi$
lies in $\DD$, since
if $\phi$ has its Denjoy--Wolff point on $\TT$, then $\phi$ has an angular derivative there,
of modulus at most 1. But this contradicts the compactness of $C_\phi$, by \cite[Cor. 3.14]{CM}.}
\end{rem}

\subsection{Applications of the semiflow model}

In this section we work with an immediately compact semigroup
$(C_{\phi_t})_{t \ge 0}$ acting on $H^2(\DD)$ or $\D$.
As in Remark \ref{rem:dw} we know that the  
Denjoy--Wolff point of the semiflow $(\phi_t)_{t \ge 0}$ lies in $\DD$, and by conjugating by the
automorphism $b_\alpha$, where
\[
b_\alpha(z):= \frac{\alpha-z}{1-\overline\alpha z},
\]
we may suppose without loss of generality that $\alpha=0$. In this case there
is a semiflow model
\[
\phi_t(z)= h^{-1}(e^{-ct}h(z)),
\]
where $c \in \CC$ with $\re c \ge 0$, and $h: \DD \to \Omega$ is a conformal
bijection between $\DD$ and a domain $\Omega \subset \CC$, with $h(0)=0$
and $\Omega$ is spiral-like  or star-like (if $c$ is real), in the sense that
\[
e^{-ct}w \in \Omega \quad \hbox{for all} \quad w \in \Omega \quad \hbox{and} \quad t \ge 0.
\]
For more details we refer to \cite{sis1,siskakis}.

Even in the case when  $(C_{\phi_t})_{t \ge 0}$ is only eventually compact, we have $\re c>0$.
Indeed, if $c=i\theta$ with $\theta \in \RR$, then there exist arbitrarily large $t>0$ such that $\theta t \in 2\pi \ZZ$,
and then $C_{\phi_t}$ is the identity mapping, and hence not compact. 

\begin{lem}\label{lem:ombound}
Let $(\phi_t)_{t \ge 0}$ be a semiflow on $\DD$ with Denjoy--Wolff point $0$. Then the
following are equivalent:\\
1. There is a $t_0>0$ with $\|\phi_{t_0}\|_\infty < 1$;\\
2. There is a $t_0>0$ with $\|\phi_{t}\|_\infty < 1$ for all $t \ge t_0$;\\
3. In the semiflow model for $(\phi_t)_{t \ge 0}$, $\re c>0$, and the domain $\Omega$ is bounded.
\end{lem}
\beginpf
1. $\Rightarrow$ 2. This follows since $\phi_t(z)=\phi_{t_0}(\phi_{t-t_0}(z))$ for all $t \ge t_0$.\\
2. $\Rightarrow$ 3. Since $\phi_t$ is not the identity mapping on $\DD$ for all $t \ge t_0$,
we have that $\re c>0$ by the argument above. Assume that $\Omega$ is unbounded; then
there is a sequence $(z_n)_n$ in $\Omega$ with $|z_n| \to \infty$; clearly also $|e^{-ct}z_n| \to \infty$
for each fixed $t \ge 0$. Since $\|\phi_{t_0}\|_\infty<1$, there exists a subsequence $(z_{n_k})_k$ of $(z_n)_n$ 
such that $(h^{-1}(e^{-ct_0}z_{n_k}))_k$ tends to $\xi\in\DD$. Therefore $e^{-ct_0}z_{n_k}\to h(\xi)$,
a contradiction since $(z_{n_k})_k$ is unbounded.\\
3. $\Rightarrow$ 1. If $M=\sup\{|z|: z \in \Omega\} < \infty$ and $\re c>0$, then for all $\epsilon>0$
we have
\[
|e^{-ct}h(z)| \le e^{-(\re c)t}M < \epsilon
\]
for $t$ sufficiently large. Then, since $h(0)=0$,   $h^{-1}$ is continuous, and
$\phi_t(z)=h^{-1}(e^{-ct}h(z))$, it follows that $\|\phi_t\|_\infty<1$ for sufficiently 
large $t$.
\endpf

We recall that a topological space is locally connected if every point has 
a neighbourhood base of connected open sets.

\begin{thm}\label{thm:appmodel}
Let $(C_{\phi_t})_{t \ge 0}$  be an immediately compact semigroup on $H^2(\DD)$ or $\D$,
such that in the semiflow model $\partial \Omega$ is locally connected.
Then the following conditions are equivalent:\\
1. There exists a $t_0 >0$ such that $\|\phi_{t_0}\|_\infty < 1$;\\
2. For all $t>0$ one has $\|\phi_t\|_\infty < 1$.\\
Therefore, if  there exists a $t_0 >0$ such that $\|\phi_{t_0}\|_\infty < 1$, then
$(C_{\phi_t})_{t \ge 0}$
is immediately trace-class.
\end{thm}

\beginpf
The only thing to prove is that 1. $\Rightarrow$ 2.

We work with the semiflow model with Denjoy-Wolff point $0$, so that
\beq\label{eq:semiflowmodel}
\phi_{t}(z)=h^{-1}(e^{-ct}h(z)),
\eeq
with $h: \DD \to \Omega$ and $c$ as above.
Assume that for some $t_1 >0$ we have $\|\phi_{t_1}\|_\infty=1$.
Then there is a sequence $(z_n)_n$ in $\DD$ with $h^{-1}(e^{-ct_1}h(z_n) ) \to e^{i\theta} \in \TT$.

Since, by Lemma \ref{lem:ombound}, $\Omega$ is bounded, there is a subsequence of $(h(z_n))_n$
converging to a point $\xi_1 \in \overline\Omega$. Moreover $\xi_1$ lies in $\partial\Omega$,
as otherwise $e^{-ct_1}\xi_1 \in \Omega$, so $h^{-1}(e^{-ct_1}\xi_1) \in \DD$, which
is a contradiction. We also have $e^{-ct_1}\xi_1 \in \partial\Omega$.

It follows that the arc $\{e^{-ct}\xi_1: 0 \le t \le t_1\}$ is contained in $\partial\Omega$, 
and since $\partial\Omega$ is locally connected, the mapping $h$ extends continuously to $\overline\DD$,
and maps $\TT$ onto $\partial\Omega$ \cite[Thm. 2.1, p. 20]{pom}.
Thus $h^{-1}\{e^{-ct}\xi_1: 0 \le t \le t_1/2\}$ is a subset of $\TT$ of positive measure, on which
$|\phi_{t_1/2}(z)|=1$. Thus $C_{\phi_{t_1/2}}$ is not compact by Proposition
\ref{prop:notimcomp},
which is a contradiction.
\endpf

We shall use similar methods to study the compactness of analytic semigroups, which is the
subject of the next subsection.

\subsection{Compact analytic semigroups}

In the particular case of analytic semigroups, the compactness is equivalent to the compactness of the resolvent,
by Theorem \ref{thm:cpresolv}, since the analyticity implies the uniform continuity \cite[p. 109]{EN}. 

\begin{rem}\label{rem:pisier}{\rm
For an analytic semigroup $(T(t))_{t\geq 0}$, being eventually compact is equivalent to be immediately compact.  Indeed, consider $Q$ the quotient map from the linear and bounded operators on a Hilbert space $\LH$, onto the Calkin algebra (the quotient of $\LH$ by the compact operators). Then $(QT(t))_{t\geq 0}$ is an analytic semigroup which vanishes for $t>0$ large enough, and therefore vanishes identically \
(see \cite{pisier}, where this observation is attributed to W.~Arendt).   }
\end{rem}

Before stating the complete characterization of compact and analytic semigroups of composition operators on $H^2(\DD)$ in terms of properties of its generator, we need the following key lemma which appears in the proof of Theorem~6.1 of \cite{siskakis}.
\begin{lem}\label{lem:sisk}
	Let $(\psi_t)_{t\geq 0}$ be a semiflow of analytic functions from $\DD$ to $\DD$, with common Denjoy--Wolff fixed point $0$, and denote by $G$ its infinitesimal generator.  Then the resolvent operator of the semigroup of composition operators 
	$(C_{\psi_t})_{t\geq 0}$ on $H^2(\DD)$ is compact if and only if 
	\[\forall \xi\in \TT, \lim_{z\to\xi,z\in\DD} \left|  \frac{G(z)}{z-\xi}  \right| =\infty.  \]
\end{lem}

\begin{thm}\label{th:compact}
 Let $G: \DD \to \CC$ be a holomorphic function such that the operator  $A$ 
defined by $Af(z)=G(z)f'(z)$ with dense domain $D(A)\subset H^2(\DD)$  
generates an analytic semigroup $(T(t))_{t\geq 0}$ of composition operators.  Then the following assertions are equivalent:
	\begin{enumerate}
		\item $(T(t))_{t\geq 0}$ is immediately compact;
		\item $(T(t))_{t\geq 0}$ is eventually compact;
		\item 
		$\forall \xi\in\TT$, $\lim_{z\in\DD,z\to\xi }\left| \frac{G(z)}{z-\xi}\right|=\infty$. 
	\end{enumerate} 
\end{thm}
\beginpf
The equivalence between 1. and 2. is given in Remark \ref{rem:pisier}. \\

We now show that 1. implies 3.
Suppose that $T(t)=C_{\varphi_t}$, where all $\varphi_t$ have a common Denjoy--Wolff fixed point $\alpha\in \overline{\DD}$.    
Assume from now that $(T(t))_{t\geq 0}$ is immediately compact. 
As in Remark~\ref{rem:dw}, it follows that $\alpha\in\DD$. In order to use Lemma~\ref{lem:sisk} we will consider another semigroup with Denjoy--Wolff point $0$. To that aim, consider the automorphism $b_\alpha(z):=\frac{\alpha-z}{1-\overline{\alpha}z}$ and $\psi_t(z):=b_\alpha\circ \varphi_t\circ b_\alpha$. Since $C_{b_\alpha}$ is invertible (equal to its inverse), and since $C_{\psi_t}=C_{b_\alpha}C_{\varphi_t}C_{b_\alpha}$, it is clear that $(T(t))_{t\geq 0}$ is immediately compact if and only if $(C_{\psi_t})_{t\geq 0}$ is immediately compact. 

Denote by $\widetilde{G}$ (resp. $G$) the generator of  the semiflow $({\psi_t})_{t\geq 0}$  (resp. $({\varphi_t})_{t\geq 0}$).  By \cite{BP}, $G(\alpha)=0$. Moreover, since  $b_\alpha\circ \psi_t=\varphi_t\circ b_\alpha$, we get
\[ \frac{|\alpha|^2-1}{(1-\overline{\alpha}\psi_t(z))^2} \frac{\partial \psi_t}{\partial t} (z)=  \frac{\partial \varphi_t}{\partial t}(b_\alpha (z)) . \]
Taking the limit as $t$ tends to $0$, we get:
\[ \frac{|\alpha|^2-1}{(1-\overline{\alpha} z)^2} \widetilde{G}(z)=  G(b_\alpha (z)), \]
and thus 
\[ G(z)=  \frac{|\alpha|^2-1}{(1-\overline{\alpha}b_\alpha (z))^2}\widetilde{G}(b_\alpha(z)).\]
It follows that
\[  \frac{(1-|\alpha|^2)}{4}|\widetilde{G}(b_\alpha(z))|\leq |G(z)|\leq \frac{(1+|\alpha|)}{(1-|\alpha|)}|\widetilde{G}(b_\alpha (z))|, \]
and then 
\[ \lim_{z\to\xi}\left|  \frac{G(z)}{z-\xi}    \right|=\infty \quad \Longleftrightarrow \quad
  \lim_{z\to\xi}\left|  \frac{\widetilde{G}(b_\alpha(z))}{z-\xi}    \right|=\infty. \]
Note that 
\[b_\alpha (z)-b_\alpha (\xi)=(z-\xi)\frac{|\alpha|^2-1}{(1-\overline{\alpha}z)(1-\overline{\alpha}\xi)},\]
with 
\[ \frac{1-|\alpha|}{1+|\alpha|}\leq \left|   \frac{(1-\overline{\alpha}z)(1-\overline{\alpha}\xi)}{|\alpha|^2-1} \right|  \leq \frac{4}{1-|\alpha|^2}.\]
Therefore we get 
\[ \lim_{z\to\xi}\left|  \frac{G(z)}{z-\xi}    \right|=\infty \quad \Longleftrightarrow \quad
   \lim_{z\to\xi}\left|  \frac{\widetilde{G}(b_\alpha(z))}{b_\alpha(z)-b_\alpha(\xi)}    \right|=\infty. \]
Using Lemma~\ref{lem:sisk} and  since $b_\alpha(\xi)\in \TT$, it follows that 1. implies 3..\\

For the implication 3. implies 1., 
we see from \cite[Thm. 1]{CDP}, that the Denjoy--Wolff point of the semigroup must belong to the unit disc, 
as otherwise there is $\tau\in \partial \mathbb D$ such that the angular limit of $\frac{G(z)}{z-\tau}$, as $z \to \tau$, is zero.
The conclusion now follows along the same lines as the previous implication.  

\endpf

Using the semiflow model, we have the following result.

\begin{thm}
Let $(C_{\phi_t})_{t \ge 0}$  be an immediately compact analytic semigroup on $H^2(\DD)$ or $\D$.
Then the following conditions are equivalent:\\
1. There exists a $t_0 >0$ such that $\|\phi_{t_0}\|_\infty < 1$;\\
2. For all $t>0$ one has $\|\phi_t\|_\infty < 1$.\\
Therefore, if  there exists a $t_0 >0$ such that $\|\phi_{t_0}\|_\infty < 1$, then
$(C_{\phi_t})_{t \ge 0}$
is immediately trace-class.
\end{thm}

\beginpf
Note that in the semiflow model, the semigroup is represented by
\eqref{eq:semiflowmodel} for all $t \in \Sigma_\alpha$: this is the correct
extension, by Corollary \ref{cor:dir} and the isolated zeroes result for analytic functions.

Following the proof of Theorem \ref{thm:appmodel}, we take a $t_1>0$ 
such that $\|\phi_{t_1}\|_\infty=1$. 
This implies that 
there exists $\xi_1 \in \partial\Omega$ such that
$e^{-cu}\xi_1 \in \partial\Omega$ for all 
$u$ in the triangle $\Sigma_\alpha \cap \{z \in \CC: \re z<t_1\}$ (cf. Lemma \ref{lem:ombound}).
This is a contradiction since a point in $\partial\Omega$ cannot have a neighbourhood consisting
of points of $\partial\Omega$.
\endpf

\subsection{Examples}
In Remark \ref{rem:pisier} we saw that whenever a semigroup is analytic,   immediate compactness is equivalent to   eventual compactness.
This is not true in general, and here is an explicit example 
showing this, based on an idea in \cite[Sec. 3]{siskakis}. 

\begin{ex}
Let $h$ be the Riemann map from $\DD$ onto the starlike region 
\[\Omega:=\DD\cup \{z\in\CC:0<\re (z) <2\mbox{ and }0<\Im(z)<1\},\]
with $h(0)=0$. Since $\partial\Omega$ is a Jordan curve, 
the Carath\'eodory theorem \cite[Thm 2.6, p. 24]{pom}
implies that $h$ extends continuously to $\partial\DD$.

Let $\phi_t(z)=h^{-1}(e^{-t}h(z))$.  Note that for $0<t<\log 2$, $\phi_t(\TT)$ intersects $\TT$ on a  set of positive measure, and thus, $C_{\phi_t}$ is not compact
by Proposition \ref{prop:notimcomp}.
Moreover, for $t>\log 2$, $\|\phi_t\|_\infty<1$, and therefore $C_{\phi_t}$ is compact (actually trace-class). Figure 1 represents the image of $\varphi_t$  for different values of $t$.  

\begin{center}
\includegraphics[width=16cm, trim = 15mm 200mm 0 0, clip  ]{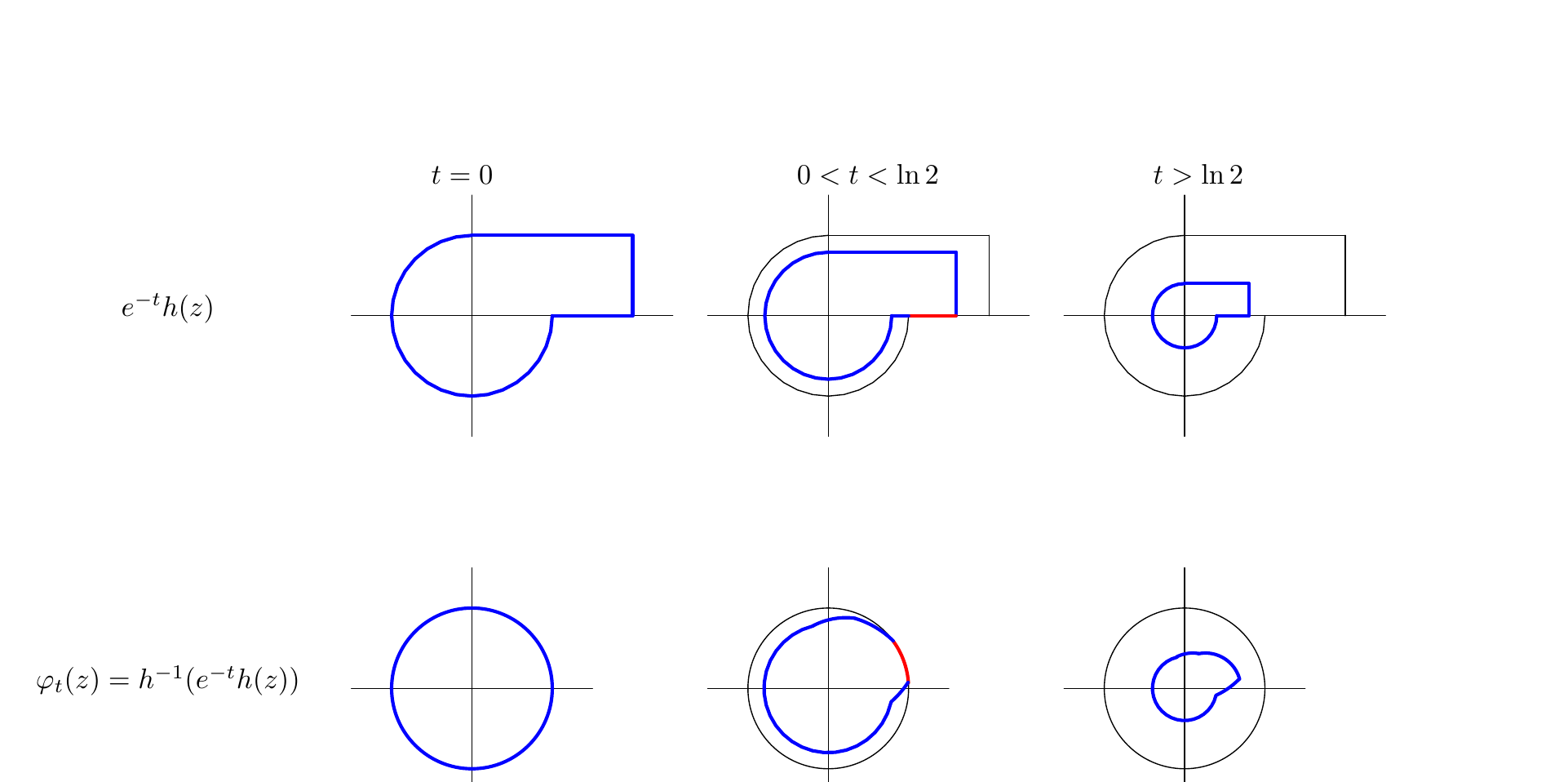}
Figure 1: an example of eventual but not immediate compactness
\end{center}
\end{ex}

It is of interest to consider the relation between immediate compactness and analyticity for
a $C_0$-semigroup of composition operators: this is because compactness of a
semigroup $(T(t))_{t \ge 0}$ is implied by compactness of the resolvent together with norm-continuity
at all points $t>0$, as in Theorem \ref{thm:cpresolv}.
\begin{ex}
Consider
\[
G(z)= \frac{2z}{z-1},
\]
Now the image of the unit circle under $\overline z G(z)$ is the line $\{z \in \CC: \re z=-1\}$, and
so the operator $A: f \mapsto Gf'$ generates a non-analytic $C_0$-semigroup 
of composition operators $(C_{\phi_t})_{t \ge 0}$ on $H^2(\DD)$.
On the other hand, it can be shown that 
$C_{\phi_t}$ is compact -- even trace-class -- for each $t>0$. For
we have the equation
\[
\phi_t(z)e^{-\phi_t(z)}=e^{-2t}ze^{-z}.
\]
Now the function $z \mapsto ze^{-z}$ is injective on $\overline\DD$; this follows from
the argument principle,  for the image of $\TT$ is easily seen to be a simple Jordan curve.
It follows that
$\|\phi_t\|_\infty<1$ for all $t>0$, and so $C_{\phi_t}$ is trace-class.
\end{ex}

\begin{ex}
The semigroup corresponding to $G(z)=(1-z)^2$ is analytic but not immediately compact.
For 
\[
\phi_t(z)=\frac{(1-t)z+t}{-tz+1+t}
\]
(note that the formula given in \cite{siskakis} contains a misprint);
the Denjoy--Wolff point is $1$, so the semigroup cannot be immediately compact.

The analyticity follows on calculating $\overline z G(z)$ for $z=e^{i\theta}$.
We obtain $-4 \sin^2(\theta/2)$, which gives the result by Corollary~\ref{cor:analytic}.
\end{ex}

\begin{ex}
Let $\varphi$ be the Riemann map from $\DD$ onto the semi-disc defined by $\{z\in\CC: \Im(z)>0, |z-1/2|<1/2 \}$ which fixes $1$. Lotto \cite{lotto} proved that $C_\phi$ is compact but not Hilbert--Schmidt. Moreover, Lotto gave an explicit formula 
for $\phi$, namely: 
\[ \phi(z)=\frac{1}{1-ig(z)}, \quad \mbox{where} \quad g(z)=\sqrt{ i\frac{1-z}{1+z}}.\]
Since  $(z^n)_n$ is an orthonormal basis of $H^2(\DD)$, it follows
that a composition operator $C_\psi$ on $H^2(\DD)$ is Hilbert--Schmidt if and only if $\displaystyle \int_0^{2\pi} \frac{1}{1-|\psi(e^{it})|^2}dt<\infty$ (see \cite{shapiro}).
It is then possible to check that $C_{\phi\circ\phi}$ is Hilbert-Schmidt, providing an example of a {\em discrete} immediately compact semigroup that is not immediately Hilbert--Schmidt, but is eventually Hilbert--Schmidt. 
\end{ex}

\section{Composition semigroups on the half-plane}
\label{sec:4}

\subsection{Quasicontractive $C_0$-semigroups}

Let $\CC_+$ denote the right half-plane in $\CC$. Berkson and Porta \cite{BP} gave the following criterion for
an analytic function $G$ to generate a one-parameter semigroup of analytic mappings from $\CC_+$ into itself,
namely, solutions to the initial value problem
\beq\label{eq:ivp}
\frac{\partial \phi_t(z)}{\partial t}= G(\phi_t(z)), \qquad \phi_0(z)=z,
\eeq
namely the condition
\beq\label{eq:xdudx}
x \frac{d (\re G)}{\partial x} \le \re G \quad \hbox{on } \CC_+,
\eeq
where as usual $x=\re z$.
Note that this does not automatically yield a $C_0$-semigroup of bounded composition operators,
since not all composition operators $C_\phi$ are bounded on $H^2(\CC_+)$.
In fact the norm of such a composition operator is finite if and only if
the non-tangential limit  $\angle \lim_{z \to \infty}\phi(z)/z$ exists   and is non-zero; we denote this by $\phi'(\infty)$
(it is positive), and
in this case $\|C_\phi\|=\phi'(\infty)^{-1/2}$. See \cite{EJ} for more details.

In fact, from the proof of \cite[Thm. 2.13]{BP} one sees that if 
the operator $A$ gven by $Af=Gf'$ generates a semigroup and
\eqref{eq:xdudx} is satisfied, then the semigroup consists of composition operators on $H^2(\CC_+)$. Moreover
Arvanitidis \cite{Arv} used the results of \cite{CDP} to
show that a necessary and sufficient condition for these composition operators to be bounded is
that the non-tangential limit  $\angle \lim_{z \to \infty} G(z)/z$   exists: if it has the value $\delta$, then
$\|C_{\phi_t}\|=e^{-\delta t/2}$, and so the semigroup is quasicontractive. We may summarize the results above as
follows.

\begin{thm}\label{thm:dec21}
For an operator $A$ given by $Af=Gf'$  on $D(A) \subset H^2(\CC_+)$, the following are equivalent.\\
(i)~$A$ generates a quasi-contractive $C_0$-semigroup of bounded composition operators on $H^2(\CC_+)$;\\
(ii)~Condition \eqref{eq:xdudx} holds and $\angle \lim_{z \to \infty} G(z)/z$ exists.
\end{thm}

A necessary condition for generation of a (quasi)contractive semigroup is given by the following result:

\begin{thm}
For an operator $A$ given by $Af=Gf'$  on $D(A) \subset H^2(\CC_+)$, if $A$ generates a  quasicontractive
$C_0$-semigroup, then 
\[
\inf_{w \in \CC_+} \frac{\re G(w) }{ \re w} > -\infty.
\]
 If the semigroup is  contractive, then $(\re G(w))/(\re w) \ge 0$ for $w \in \CC_+$.
\end{thm}
\beginpf
Let $k_w$ be the reproducing kernel for $H^2(\CC_+)$
given by
\[
k_w(z)= \frac{1}{2\pi} \frac{1}{z+\overline w} \qquad (z,w \in \CC_+)
\]
(cf. \cite[p. 8]{par04}).
Then
\[
 \re \frac{\langle Ak_w, k_w \rangle}{ \langle k_w,k_w \rangle}= -\frac{\re G(w) }{2 \re w},
\]
and the result follows immediately from the Lumer--Phillips theorem.
\endpf

For analytic semigroups, we have the following necessary condition.


\begin{prop}
Suppose that $A:f \to Gf'$ generates an analytic semigroup on $H^2(\CC_+)$. Write $G=u+iv$ where $u$ and $v$ are real functions
and similarly $z=x+iy$. Then
there is an $\alpha$ with $0< \alpha < \pi/2$ such that, for every fixed $y \in \RR$,
$(u \cos \theta + v \sin \theta)/x$ is a decreasing function of $x$ 
for all $\theta \in [-\alpha,\alpha]$.
\end{prop}
\beginpf
Note that  the criterion \eqref{eq:xdudx} may be rewritten as 
\[
\frac{\partial}{\partial x} \left( \frac{u}{x} \right) \le 0.
\]
The  result now follows immediately on applying this to the functions $Ge^{-i\theta}$, which generate  $C_0$-semigroups.
\endpf

\subsection{Groups of composition operators on the half-plane}

It turns out that there are very few groups of composition operators on the half-plane.
\begin{prop}
Suppose that $A:f \to Gf'$ generates a $C_0$ quasicontractive group 
of bounded composition operators on $H^2(\CC_+)$. 
Then $G(z)=pz+iq$ for some real $p$ and $q$, and hence, for $t\in \RR,z\in \CC_+$, we get 
\beq\label{eq:solphit}
\phi_t(z)= ze^{pt}+\frac{iq}{p} \left( e^{pt}-1 \right)  
\eeq
if $p\neq 0$, and 
\beq\label{eq:2solphit}
\phi_t(z)= z+iqt
\eeq
if $p= 0$. 
\end{prop}
\beginpf
Once again we begin  with condition \eqref{eq:xdudx}, applying it to $G$ and $-G$, to obtain
\[
x \frac{\partial u}{\partial x}=u.
\]
The solution to this is $u=F(y)x$ for some smooth real function of $y$. The Cauchy--Riemann equations
imply that $\dfrac{\partial v}{\partial y}=F(y)$; that is, $v=\int F \, dy + E(x)$ for some function $E$. Likewise,
\[
\frac{\partial u}{\partial y}=F'(y)x=-E'(x),
\]
and thus $F(y)=ay+b$ and $E(x)=-ax^2/2+c$ for some real constants $a$, $b$ and $c$.
We conclude that $G(z)=  - iaz^2/2+bz-ic$. Now Theorem~\ref{thm:dec21} implies that $a=0$ and the result 
for $G$ follows.
It is now clear from \eqref{eq:ivp} that $\phi_t$ is as given in \eqref{eq:solphit} or \eqref{eq:2solphit}.
\endpf

\begin{ex}
Taking $G(z)=pz+iq$ with $p,q\in \RR$ and $q\neq 0$,  we obtain a group which is not analytic. 

Moreover, taking $G(z)=1-z$ (so that $\phi_t(z)=e^{-t}z+1-e^{-t})$, we get an example of $C_0$-semigroup which is neither a group nor analytic.
\end{ex}

\begin{rem}{\rm
	The question of characterizing the compact semigroups is not relevant, since no composition operator on the Hardy space of the half-plane is  compact \cite[Cor. 3.3]{EJ}.  }
\end{rem}

\section*{Acknowledgement}
The authors are grateful to the referee for several comments allowing them to improve the
paper. In particular, the third condition of Theorem \ref{th:compact} has been simplified.


\begin{thebibliography}{99}

\bibitem{ABHN}
W. Arendt,  C.J.K. Batty, M. Hieber and F. Neubrander, {\em Vector-valued Laplace transforms and Cauchy problems}. Monographs in Mathematics, 96. Birkh\"auser Verlag, Basel, 2001.

\bibitem{arendt}
W. Arendt and A.F.M. ter Elst, From forms to semigroups. {\em Spectral theory, mathematical system theory, evolution equations, differential and difference equations}, 47--69, Oper. Theory Adv. Appl., 221, Birkh\"auser/Springer Basel AG, Basel, 2012.

\bibitem{Arv}
A.G. Arvanitidis, Semigroups of composition operators on Hardy spaces of the half-plane,  
{\em Acta Sci. Math. (Szeged)\/} 
81 (2015), no. 1--2, 293--308.

\bibitem{ACP}  C. Avicou, I. Chalendar and J.R. Partington, A class of quasicontractive semigroups acting on Hardy and Dirichlet space, 
{\em J. Evol. Equ.} 15 (2015), no. 3, 647--665.

\bibitem{berkson}
E. Berkson,  
Composition operators isolated in the uniform operator topology. 
{\em Proc. Amer. Math. Soc.} 81 (1981), no. 2, 230--232.

\bibitem{BP} E. Berkson and H. Porta,   \emph{Semigroups of analytic functions and composition operators.} Michigan Math. J. 25 (1978), no. 1, 101--115.

\bibitem{CDP}
M.D. Contreras, S.  D\'\i az Madrigal and Ch. Pommerenke,   On boundary critical points for semigroups of analytic functions. {\em Math. Scand.} 98 (2006), no. 1, 125--142.

\bibitem{CM} C.C. Cowen and B.D. MacCluer,   \emph{Composition operators on spaces of analytic functions.} Studies in Advanced Mathematics. CRC Press, Boca Raton, FL, 1995.


\bibitem{primer}
O. El-Fallah, K.  Kellay, J. Mashreghi and T. Ransford, 
{\em A primer on the Dirichlet space}. Cambridge Tracts in Mathematics, 203. 
Cambridge University Press, Cambridge, 2014.

\bibitem{EJ}
S. Elliott and M.T.  Jury,   Composition operators on Hardy spaces of a half-plane. 
{\em Bull. Lond. Math. Soc.} 44 (2012), no. 3, 489--495.

\bibitem{EN}
 K.J. Engel and R. Nagel, {\em A short course on operator semigroups.} Springer, 2005.
 
 \bibitem{koenig}
W. K\"onig, Semicocycles and weighted composition semigroups on $H^p$. 
{\em Michigan Math. J.} 37 (1990), 469--476.  


\bibitem{lotto}
B. A. Lotto, A compact composition operator that is not Hilbert--Schmidt, 
{\em Studies on Composition Operators}, Contemporary Mathematics, 213,  Amer. Math. Soc., Rhode Island, 1998, 93--97.
 
\bibitem{rosenthal}
R.A. Mart\'\i nez-Avenda\~no and P. Rosenthal,   
{\em An introduction to operators on the Hardy--Hilbert space}. Graduate Texts in Mathematics, 237. Springer, New York, 2007.

\bibitem{par04}
J.R. Partington, {\em Linear operators and linear systems}. London Mathematical Society Student Texts, 60. Cambridge University Press, Cambridge, 2004.

\bibitem{pazy}
A. Pazy, {\em Semigroups of linear operators and applications to partial differential equations}. Applied Mathematical Sciences, 44. Springer-Verlag, New York, 1983.

\bibitem{pisier}
G. Pisier, {\em A remark on hypercontractive semigroups and operator ideals},
preprint, 2007.
{\tt http://arxiv.org/abs/0708.3423}.

\bibitem{pom}
Ch. Pommerenke, {\em Boundary behaviour of conformal maps}. Grundlehren der Mathematischen Wissenschaften, 299. Springer-Verlag, Berlin, 1992.




\bibitem{shapiro}
J.H. Shapiro and P.D. Taylor,  Compact, nuclear, and Hilbert-Schmidt composition operators on $H^2$. {\em Indiana Univ. Math. J.}
23 (1973/74), 471--496.

\bibitem{sis1}
A.G. Siskakis, Semigroups of composition operators on the Dirichlet space. Results Math. 30 (1996), no. 1--2, 165--173. 

\bibitem{siskakis} A.G. Siskakis,  Semigroups of composition operators on spaces of analytic functions, a review.
{\em Studies on composition operators\/} (Laramie, WY, 1996), 229--252, Contemp. Math., 213, Amer. Math. Soc., Providence, RI, 1998. 

\bibitem{SF} G.A. Sviridyuk and V.E. Fedorov,  
\emph{Linear Sobolev type equations and degenerate semigroups of operators.}
Inverse and Ill-posed Problems Series. VSP, Utrecht, 2003.

\end{thebibliography}
\end{document}